\documentclass[11pt]{amsart}
\usepackage{mathrsfs}
\usepackage{threeparttable}
\usepackage{amsfonts}
\usepackage{amsmath}
\usepackage{amssymb}
\usepackage{enumerate}
\usepackage{hyperref}
\usepackage{latexsym}
\usepackage{color}
\usepackage{enumitem}
\usepackage{enumerate}
\usepackage{makecell}
\usepackage{tikz}
\usepackage{changes}
\usepackage{listings}

\def\a{\alpha}
\def\b{\beta}
\def\S{\mathrm{S}} 
\def\l{\langle} \def\r{\rangle}

\def\Ga{\Gamma}

\def\qed{\hfill $\Box$}

 \newcommand\D{\mathrm{D}}
\newcommand\A{\mathrm{A}} \newcommand\Sy{\mathrm{S}}
\newcommand\Aut{\mathrm{Aut}}  
\newcommand\Cay{\mathrm{Cay}}  \newcommand\Cos{\mathrm{Cos}} \newcommand\K{\mathsf{K}}

\def\mz{{\mathbb Z}}

\newcommand\GL{\mathrm{GL}} \newcommand\SL{\mathrm{SL}}  \newcommand\PGL{\mathrm{PGL}}    \newcommand\PSL{\mathrm{PSL}}

\newcommand\PGammaL{\mathrm{P\Gamma L}}

\newcommand\M{\mathrm{M}}

\newtheorem{theorem}{Theorem}[section]

\newtheorem{lemma}[theorem]{Lemma}

\newtheorem{proposition}[theorem]{Proposition}
\theoremstyle{definition}
\newtheorem{example}[theorem]{Example}
\newtheorem{problem}[theorem]{Problem}

\def\pf{\noindent{\it Proof.}}

  \def\char{\,\hbox{\rm char}\,}

\begin{document}
\title[Cover graphs]{On Arc-Transitive Regular Covers of Cubic Edge-Primitive Graphs}
\thanks{Corresponding author: Hao Yu}
\thanks{This work was supported by NNSFC (12571362), NSF of Guangxi (2025GXNSFAA069013) and the Special Foundation for Guangxi Ba Gui Scholars}
\author[F. Deng, J.J. Li, Y. Wang \and H. Yu]{%
Feng Deng, Jing Jian Li, Yu Wang and Hao Yu}
\address{Guangxi Center for Mathematical Research   \\
\& Center for Applied Mathematics of Guangxi (Guangxi University) \\
Guangxi University \\
Nanning, Guangxi 530004, P.R. China.}
\email{dengfeng@st.gxu.edu.cn (F. Deng); \allowbreak lijjhx@gxu.edu.cn (J.J. Li);
\allowbreak   haoyu@gxu.edu.cn (H. Yu).}

\address{Center for Combinatorics and LPMC, Nankai University, Tianjin, 300071, P.R. China}
\email{wangyu972022@163.com (Y. Wang)}

\begin{abstract}
We determine, up to isomorphism of the covering graphs, the connected arc-transitive regular covers of cubic edge-primitive graphs whose covering transformation group is cyclic or elementary abelian of order $p^2$, where $p$ is a prime. Combining the known classifications for the base graphs ${\rm K_{3,3}}$ and ${\rm DC_{14}}$ with new arguments for ${\rm F30A}$ and ${\rm F102A}$ gives the full list in these two classes of covering groups. In the cyclic case, the covers of ${\rm F30A}$ and ${\rm F102A}$ are ${\rm F90A}$ and ${\rm F204A}$, respectively. In the elementary abelian case, neither ${\rm F30A}$ nor ${\rm F102A}$ admits an arc-transitive regular $\mathbb{Z}_p^2$-cover, so the base graph is ${\rm K_{3,3}}$ or ${\rm DC_{14}}$.

\vskip 5pt 


\end{abstract}
\maketitle

\parskip 5pt
\section{Introduction}
All graphs considered in this paper are finite, simple, undirected, and connected. 
For a graph $\Ga$, we denote by $V(\Ga)$, $E(\Ga)$, ${\rm Arc}(\Ga)$ and $\Aut(\Ga)$, its vertex set, edge set, arc set and full automorphism group, respectively.

For a positive integer $s$, an {\it $s$-arc} in $\Ga$ is a sequence of vertices $(\a_0,\,\a_1,\,\dots,\,\a_s)$ such that $\{\a_i,\,\a_{i+1}\}\in E(\Ga)$ for $0\leqslant i\leqslant s-1$, and $\a_{j-1}\ne\a_{j+1}$ for $1\leqslant j\leqslant s-1$. 
The graph $\Ga$ is said to be \textit{$(X,s)$-arc-transitive} if $X$ acts transitively on both the vertices and the $s$-arcs of $\Ga$, where $X\leqslant\Aut(\Ga)$. 
It is called \textit{$(X,s)$-transitive} if it is $(X,s)$-arc-transitive but not $(X,s+1)$-arc-transitive. 
For $s=1$, $\Ga$ is simply called \textit{$X$-arc-transitive} (also called \textit{$X$-symmetric}). 
The graph $\Ga$ is called \textit{$X$-edge-primitive} if $X$ acts transitively on the edges of $\Ga$ and the stabilizer $X_{\{\a,\,\b\}}$ of some edge $\{\a,\,\b\}\in E(\Ga)$ is a maximal subgroup of $X$. 
The graph is called \textit{$s$-arc-transitive} (\textit{edge-primitive}) if $X=\Aut(\Ga)$.
An arc-transitive graph $\Ga$ is said to be \textit{$(X,s)$-arc-regular} if $X$ acts regularly on the set of all $s$-arcs of $\Ga$. The graph is called \textit{$s$-arc-regular} if $X=\Aut(\Ga)$.
Following the pioneering work of Tutte \cite{T1947}, cubic graphs with specific symmetry properties have been extensively studied. 
For an $s$-arc-regular cubic graph $\Ga$, Djokovi\'{c} and Miller \cite[Propositions 2–5]{DM1980} showed that the stabilizer in $\Aut(\Ga)$ of a vertex $v\in V(\Ga)$ is isomorphic to $\mz_3,\,\Sy_3,\,\Sy_3\times\mz_2,\,\Sy_4$ or $\Sy_4\times\mz_2$ for $s=1,\,2,\,3,\,4,\,5$, respectively. 
Based on this result, Conder and Potočnik \cite{CP2026} conducted an exhaustive computer search to produce a complete list of $s$-arc-regular cubic graphs on up to 10000 vertices.

Let $\Ga$ be a graph, let $X\leqslant\Aut(\Ga)$, and let $K\unlhd X$. Suppose that $K$ is semiregular on $V(\Ga)$. The \emph{normal quotient graph} $\Ga_K$ is the graph whose vertices are the $K$-orbits on $V(\Ga)$, with two distinct $K$-orbits adjacent whenever there is an edge of $\Ga$ joining them. The graph $\Ga$ is called a \emph{regular cover}, or a \emph{$K$-cover}, of a graph $\Sigma$ if $\Ga_K\cong\Sigma$ and, for every pair of adjacent $K$-orbits $B$ and $C$, each vertex of $B$ has exactly one neighbour in $C$. In this case, $K$ is called the \emph{covering transformation group}, the $K$-orbits on $V(\Ga)$ are called the \emph{fibres}. Equivalently, for each $\a\in V(\Ga)$, the quotient map induces a bijection from $N(\a)$ onto $\Ga_K(\a^K)$, where $N(\a)$ denotes its neighbourhood in $\Ga$. In particular, $\Ga$ and $\Ga_K$ have the same valency. 
In this paper, two regular covers are identified whenever their covering graphs are isomorphic. Thus the classification below concerns covering graphs up to graph isomorphism and does not distinguish equivalent or inequivalent covering projections onto the same base graph.
If $K$ is cyclic, the cover $\Ga$ is called a \textit{cyclic cover} of $\Sigma$.
An automorphism of $\Ga$ is said to be \textit{fibre-preserving} if it permutes the fibres. All such automorphisms form a group, called the \textit{fibre-preserving group}. 
Regular covers of small order graphs have received considerable attention. For instance, the 3-dimensional cube graph \cite{FW2003}, the Heawood graph \cite{MMP20041}, ${\rm K_4}$ \cite{FK2007, FK20041}, ${\rm K_{3,3}}$ \cite{FK2007, FK20042}, Petersen graph \cite{MP2006}, the Pappus graph \cite{O2009}, and etc. 
The regular covers of the complete graph ${\rm K_n}$ with 2-arc-transitive fibre-preserving automorphism groups have been classified when the covering transformation group is cyclic \cite{DMW1998}, isomorphic to $\mz_p^2$ \cite{DMW1998}, metacyclic \cite{XDKX2015}, or isomorphic to $\mz_p^3$ \cite{DKX2005} (where $p$ is a prime). For further related work, readers may refer to the literature \cite{FK2006, FKX2006, FKW2005, KO2009, MMP20041, PHL2015}.

The classification of cubic edge-primitive graphs is known \cite{W1973}. We consider their arc-transitive regular covers when the covering transformation group is cyclic or isomorphic to $\mz_p^2$ for some prime $p$. The cases with base graph ${\rm K_{3,3}}$ or ${\rm DC_{14}}$ follow from the classifications in \cite{FK20042,WH2010,DXY2018,CM20132}. The new part of the present paper is the determination of the cyclic covers with base graph ${\rm F30A}$ or ${\rm F102A}$ and the proof that neither of these two graphs admits an arc-transitive regular $\mz_p^2$-cover. Combining these results gives the two classification theorems below. The notation used in their statements is introduced in the next section.

\begin{theorem}\label{result 1}
Let $\Ga$ be a connected $X$-arc-transitive $K$-cover of a cubic edge-primitive graph $\Sigma$, where $X$ is the fibre-preserving group and $K\cong\mz_n$ with $n\geqslant2$. Then $\Sigma\cong{\rm K_{3,3}}$, ${\rm DC_{14}}$, ${\rm F30A}$ or ${\rm F102A}$, and one of the following holds.
\begin{itemize}
\item [(1)] If $\Sigma\cong{\rm K_{3,3}}$, then $\Ga\cong G(k,n)$ for some $k$ satisfying $2\leqslant k<n$ and $k^2+k+1\equiv0\pmod n$, or $n=3$ and $\Ga\cong{\rm F18A}$.
\item [(2)] If $\Sigma\cong{\rm DC_{14}}$, then either $\Ga\cong\mathcal{D}(1,7n,\lambda)$ with $\lambda^2+\lambda+1\equiv0\pmod{7n}$, or $7\mid n$ and $\Ga\cong\mathcal{D}(7,n/7,\mu)$, where $\mu=0$ for $n=7$ and $\mu^2+\mu+1\equiv0\pmod{n/7}$ for $n>7$.
\item [(3)] If $\Sigma\cong{\rm F30A}$, then $n=3$ and $\Ga\cong{\rm F90A}$.
\item [(4)] If $\Sigma\cong{\rm F102A}$, then $n=2$ and $\Ga\cong{\rm F204A}$.
\end{itemize}
\end{theorem}

Next, we consider regular covers of cubic edge-primitive graphs with covering transformation group $\mz_p^2$ ($p$ prime).

\begin{theorem}\label{result 2}
Let $\Ga$ be a connected $X$-arc-transitive $K$-cover of a cubic edge-primitive graph $\Sigma$, where $X$ is the fibre-preserving group and $K\cong\mz_p^2$ for some prime $p$. Then $\Sigma\cong{\rm K_{3,3}}$ or ${\rm DC_{14}}$, and the following statements hold.
\begin{itemize}
\item [(1)] If $\Sigma\cong{\rm K_{3,3}}$, then $\Ga\cong{\rm F24A}$, ${\rm F54A}$, $X(3)$ or ${\rm K_{3,3}}\times_{\ell^i}\mz_p^2$ with $i=1,2$ and $p\geqslant5$.
\item [(2)] If $\Sigma\cong{\rm DC_{14}}$, then $\Ga\cong\mathcal{H}_p$, or $p=7$ and $\Ga\cong{\rm F686C}$.
\end{itemize}
\end{theorem}

After this introductory section, the necessary notation, the examples of the graphs arising in Theorems \ref{result 1} and \ref{result 2}, and other useful facts will be given in the next section, and Theorems \ref{result 1} and \ref{result 2} will be proved in Sections \ref{sec3} and \ref{sec4}, respectively.

\section{Preliminaries}\label{sec2}

The notation ${\rm FnA}$, ${\rm FnB}$, etc., refer to the corresponding graphs of order $n$ in the Foster census of cubic symmetric graphs \cite{CP2026}. For definitions and notations not explicitly stated here, we refer the reader to \cite{XD2018}. 

{\bf 2.1. Examples}

To state the main theorems, we first introduce some infinite families of covers of ${\rm K_{3,3}}$ whose covering transformation groups are either $\mz_n$ or $\mz_p^2$ for some integer $n>1$ and prime $p$. 
Let $\Sigma={\rm K_{3,3}}$ be the complete bipartite graph with bipartition $\{1,\,3,\,5\}\cup\{2,\,4,\,6\}$, whose full automorphism group is isomorphic to $\Sy_3\wr\mz_2$. 

\begin{example}\cite[Page 103]{FK20042}
Let $k$ and $n$ be two non-negative integers. 
We define a family of derived graphs $G(k,n):={\rm K_{3,3}}\times_h\mz_n$. Here, the voltage assignment $h: {\rm Arc(K_{3,3})}\to \mz_n$ satisfies $h_{4,1}={h_{1,4}}^{-1}={h_{6,3}}^{-1}=h_{3,6}=k+1, \quad h_{4,5}={h_{5,4}}^{-1}={h_{6,1}}^{-1}=h_{1,6}=1$ and $h_{u,v}=0$ for all other arcs $(u,v)$ of ${\rm K_{3,3}}$.
\end{example}

\begin{example}\cite[Equation (1) and Proposition 2.4]{AP2018}\label{def:DCfamily}
Let $m$ and $q$ be positive integers, and define $D(m,q)=\l a, b, c\mid a^2=b^{mq}=c^m=1, aba=b^{-1}, aca=c^{-1}, bc=cb\r$. Take $\lambda=0$ for $q=1$, and for $q>1$ let $\lambda$ satisfy $\lambda^2+\lambda+1\equiv0\pmod q$. Define $\mathcal{D}(m,q,\lambda)=\Cay(D(m,q), \{a, ab, ab^{-\lambda}c\})$. For $m=1$, the generator $c$ is trivial and $D(1,q)\cong\D_{2q}$. In particular, $\mathcal{D}(1,7,2)\cong\mathcal{D}(1,7,4)\cong{\rm DC_{14}}$.
\end{example}

The three elements in the connection set are involutions and generate $D(m,q)$, so $\mathcal{D}(m,q,\lambda)$ is connected and cubic.

\begin{example}\cite[Sections 2--6]{CM20132}\label{def:Hp}
Let $\mathcal{U}=\l h,a\mid h^3=a^2=1\r$, $r=[h,a]$ and $s=[h^{-1},a]$, and let $N$ be the normal subgroup of index $42$ in $\mathcal{U}$ freely generated by $w_1=[r,s^{-1}]$, $w_2=rsr$, $w_3=[s,r^{-1}]$, $w_4=s^2r^{-1}s$, $w_5=r^2s$, $w_6=(rs^{-1})^2r$, $w_7=sr^2$ and $w_8=sr^{-1}s^2$. Let $v_1=w_1$, $v_2=w_2w_7^{-1}$, $v_3=w_3$, $v_4=w_4w_8^{-1}$, $v_5=w_5w_7^{-1}$ and $v_6=w_6w_7^{-1}w_8$. For a prime $p$, put $L_p=\l N',v_1,\ldots,v_6,w_1^p,\ldots,w_8^p\r$. Then $L_p\unlhd\mathcal{U}$ and $N/L_p\cong\mz_p^2$. Let $\mathcal{U}_p=\mathcal{U}/L_p$ and $H_p=\l hL_p\r$, and define $\mathcal{H}_p=\Cos(\mathcal{U}_p,H_p,H_p(aL_p)H_p)$.
\end{example}

By \cite[Theorem 7.1]{CM20132}, $\mathcal H_p$ is a connected arc-transitive cubic $\mz_p^2$-cover of ${\rm DC_{14}}$. If $p\neq7$, then $\mathcal H_p$ is the unique such cover and is $1$-arc-regular. For $p=7$, the two covers are ${\rm F686B}$ and ${\rm F686C}$, which are $2$-arc-regular and $1$-arc-regular, respectively. Moreover, $\mathcal H_2\cong{\rm F56A}$ and $\mathcal H_7\cong{\rm F686B}$.

We now describe three infinite families of covers. In each family, $K$ is identified with the additive group of the two-dimensional vector space $V(2,p)$ over the finite field $\mathbb{F}_p$, where $p$ is prime.

\begin{example}\cite[Page 806(3)]{DXY2018}\label{cover1}
Let $X(3)={\rm K_{3,3}}\times_f\mz_p^2$ be a family of derived graphs, where the voltage assignment $f: {\rm Arc}({\rm K_{3,3}})\to \mz_p^2$ satisfies 
\[
f_{1,4}=(0,1)=f_{4,1}^{-1}, \quad f_{1,6}=(1,0)=f_{6,1}^{-1}, \quad f_{5,2}=(0,-1)=f_{2,5}^{-1}, \quad f_{5,4}=(-1,0)=f_{4,5}^{-1},
\] 
and $f_{u,v}=0$ otherwise, for each prime $p\geqslant5$.
\end{example}

\begin{example}\cite[Page 79, Table 2.1]{KN2007}\label{cover2}
Define two families of derived graphs ${\rm K_{3,3}}\times_{\ell^i}\mz_p^2$, for $i=1,\,2$ and prime $p\geqslant 5$, where the voltage assignments $\ell^i: \mathrm{Arc}(K_{3,3})\to \mz_p^2$ are given by Figure \ref{K_{3,3}}, with nonzero voltages specified by 
\[
(x_1^T,x_2^T,x_3^T,x_4^T)=
\begin{bmatrix}
1 & 1 & 1 & -2  \\
1 & -1 & 1 & 0
\end{bmatrix}
\mbox{ or }
\begin{bmatrix}
1 & 0 & -1 & 0  \\
0 & 1 & 1 & 0
\end{bmatrix}
\]
respectively, where $T$ denotes the transpose of a vector.
\end{example}

\begin{figure}[ht]
  \centering
\begin{tikzpicture}
\tikzstyle{main node}=[draw, shape=circle, label distance=-0.5mm, inner sep=1pt];
\node[main node] (2) at (0,2) [label=above:\tiny{$2$}] {};
\node[main node] (4) at (2,2) [label=above:\tiny{$4$}] {};
\node[main node] (6) at (4,2) [label=above:\tiny{$6$}] {};
\node[main node] (1) at (0,0) [label=below:\tiny{$1$}] {};
\node[main node] (3) at (2,0) [label=below:\tiny{$3$}] {};
\node[main node] (5) at (4,0) [label=below:\tiny{$5$}] {};
\draw[line width=1pt] (6)--(1)--(2)--(3) (5)--(6)--(3)--(4)--(5)--(2)(4)--(1);
\draw[line width=1pt, -stealth] (1)--(0.75,0.75) node[midway, above] {\tiny{$x_1$}};
\draw[line width=1pt, -stealth] (2)--(1,1.5) node[midway, above] {\tiny{$x_2$}}; 
\draw[line width=1pt, -stealth] (3)--(2.5,0.5) node[midway, right] {\tiny{$x_3$}};
\draw[line width=1pt, -stealth](6)--(3,1.5) node[midway, above] {\tiny{$x_4$}};
\end{tikzpicture}
\caption{${\rm K_{3,3}}$ with voltage assignments}\label{K_{3,3}}
\end{figure}
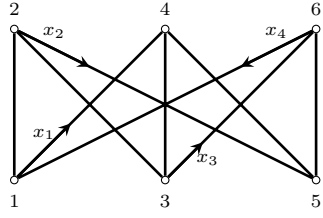

Before presenting covers of others, we called several important cubic symmetric graphs as follows. 
The Heawood graph ${\rm DC_{14}}$, with full automorphism group isomorphic to $\PGL(2,7)$, is 4-arc-regular.
The Pappus graph, denoted ${\rm F18A}$, is the unique cubic symmetric graph on 18 vertices. It is a bipartite 3-arc-regular graph with 27 edges, and its full automorphism group is isomorphic to $\mz_3.(\Sy_3\wr\mz_2)$.
The Tutte-Coxeter graph ${\rm F30A}$ is 5-arc-regular, whose full automorphism group is isomorphic to $\PGammaL(2,9)$.
The Biggs-Smith graph ${\rm F102A}$ is a cubic distance-transitive graph whose full automorphism group is isomorphic to $\PSL(2,17)$, and it is 4-arc-regular.

\begin{example}\label{90V}
Let $\Ga\cong{\rm F90A}$ and let $A=\Aut(\Ga)$. By \cite[Web]{CP2026}, $A\cong\mz_3.\A_6.\mz_2^2$, and so $A$ contains a normal subgroup $K\cong\mz_3$. Since $\Ga$ is a connected arc-transitive cubic graph, it is $A$-vertex-transitive and locally primitive. The subgroup $K$ has at least three orbits on $V(\Ga)$, so \cite[Lemma 2.5]{LP2008} shows that $K$ is semiregular and $\Ga$ is a regular $K$-cover of $\Ga_K$. Thus $\Ga_K$ is a connected arc-transitive cubic graph of order $30$. By \cite[Web]{CP2026}, the unique such graph is ${\rm F30A}$. Hence ${\rm F90A}$ is a regular $\mz_3$-cover of ${\rm F30A}$.
\end{example}

Let $\Sigma$ be a finite simple graph. The \textit{ standard double cover} of $\Sigma$ is the graph $\Sigma\times {\rm K_2}$, where $\times$ denotes the direct product of graphs. 
Explicitly, the vertex set is $V(\Sigma)\times\{1,\,2\}$, and two vertices $(u,i)$ and $(v,j)$ are adjacent if and only if $u$ and $v$ are adjacent and $i\neq j$.
The graph $\Sigma \times {\rm K_2}$ is bipartite with bipartite halves $V(\Sigma)\times\{1\}$ and $V(\Sigma)\times\{2\}$. 

\begin{example} \label{204V}
Let $\Ga={\rm F102A}\times{\rm K_2}$ be the direct product of graphs. 
Then $\Ga$ is cubic symmetric graph of order $204$ and \cite[Web]{CP2026} shows that  $\Ga\cong {\rm F204A}$. Hence $\Ga$ is exactly the standard double cover of ${\rm F102A}$.
\end{example}

{\bf 2.2. Graph and group theoretic results}

An important method for studying vertex-transitive graphs is taking normal quotient graphs. 
A graph $\Ga$ is said to be \textit{$X$-locally primitive} if for every vertex $v$, the stabilizer $X_v$ acts primitively on $N(v)$. 
It is known that every vertex-transitive locally primitive graphs is a regular cover of some basic locally primitive graph \cite{PLN1996}. 
This fact shows that the classification of locally primitive graphs can be reduced to the study of regular covers of their base graphs. 
Noting that an edge-transitive graph with prime valency is locally primitive. 
The following proposition provides a basic method for studying vertex-transitive locally-primitive graphs, which was first proved by Praeger \cite[Theorem 4.1]{P1993} for 2-arc-transitive graphs and slightly generalized to locally-primitive graphs.

\begin{proposition}\cite[Lemma 2.5]{LP2008}\label{NQ}
Let $\Ga$ be a connected $X$-vertex-transitive locally-primitive graph, and let $K\unlhd X$ have at least three orbits on $V(\Ga)$. Then the following statements hold. 
\begin{itemize}
\item [(1)] $K$ is semiregular on $V(\Ga)$, $X/K\leqslant\Aut(\Ga_K)$, $\Ga_K$ is $X/K$-locally-primitive, and $\Ga$ is an $K$-cover of $\Ga_K$; 
\item [(2)] $\Ga$ is $(X,s)$-arc-transitive if and only if $\Ga_K$ is $(X/K,s)$-arc-transitive, where $1\leqslant s\leqslant 5$ or $s=7$;
\item [(3)] $X_\a\cong (X/K)_v$, where $\a\in V(\Ga)$ and $v\in V(\Ga_K)$.
\end{itemize}
\end{proposition}

Let $X$ be a group. Recall that $X$ is called an \textit{extension} of $K$ by $G$, denoted by $X=K.G$, if $K$ is a normal subgroup of $X$ such that the quotient group $X/K$ is isomorphic to $G$. 
Furthermore, $X=K.G$ is called a \textit{central extension} if $K\leqslant Z(X)$ and a \textit{proper central extension} if $K\leqslant Z(X)\cap X'$. 
A group $X$ is said to be \textit{perfect} if $X=X'$, the commutator subgroup. 
For a given group $G$, if $X=K.G$ is perfect and maximal order central extension of $G$, then $K$ is called the \textit{Schur multiplier} of $G$, denoted by ${\rm Mult}(G)$. 
The Schur multipliers of all finite simple groups are explicitly known, see \cite[Page 302]{G1982}. 
The following proposition is known.

\begin{proposition}\cite[Lemma 2.5]{PHD2017}\label{simple}
Assume that $X=K.T$, where $K$ is a cyclic group, and $T$ is a nonabelian simple group. 
Then $X=K.T$ is a central extension, $X=KX'$, and $X'=(K\cap X').T$ is a perfect group, where $K\cap X'$ is isomorphic to a subgroup of the Schur multiplier of $T$.
\end{proposition}

\begin{proposition}\cite[Page XV]{C1985}\label{schur}
The Schur multiplier ${\rm Mult}(\PSL(2,q))$ of the nonabelian simple group $\PSL(2,q)$ with $q$ an odd prime power is $\mz_2$ for $q\neq9$, and $\mz_6$ for $q=9$.
\end{proposition}

The next proposition is the characterization of subgroups of $\GL(2,p)$.
\begin{proposition}\label{GL}\cite[Lemma 2.7]{DMW1998}
The general linear group $\GL(2,p)$ with $p$ a prime contains neither a nonabelian simple subgroup nor a subgroup isomorphic to $\A_4$.
\end{proposition}

\section{Cyclic Cover}\label{sec3}

In this section, let $\Ga$ be a connected $X$-arc-transitive cyclic cover of the edge-primitive cubic graph $\Sigma$, where $X$ is the fibre-preserving group and the covering transformation group is $K\cong\mz_n$ for some integer $n\geqslant2$.
Since $X$ is the fibre-preserving group, we have $K\unlhd X$. Let $Y=X/K$.
Then $Y\leqslant\Aut(\Sigma)$, $X=K.Y$ and $X_\a\cong Y_v$ for any $\a\in V(\Ga)$ and $v\in V(\Sigma)$.
By \cite[Lemma 2.9]{W1973}, $\Sigma\cong {\rm K_{3,3}},\,{\rm DC_{14}},\,{\rm F30A}$ or ${\rm F102A}$, all of which are vertex-transitive locally primitive.
Since cyclic covers of ${\rm K_{3,3}}$ and ${\rm DC_{14}}$ have been studied in \cite[Theorem 1.1]{FK20042} and \cite[Theorem 1.1]{WH2010}, it suffices to focus on cyclic covers of the remaining two edge-primitive cubic graphs. We deal it with the following two lemmas.

\begin{lemma}\label{F30A}
Suppose that $\Sigma\cong {\rm F30A}$. Then $K\cong\mz_3$ and $\Ga\cong {\rm F90A}$.
\end{lemma}
\pf\, It is well known that ${\rm F30A}$ is a bipartite graph and $\Sigma$ is 5-arc-regular under $\Aut(\Sigma)\cong\PGammaL(2, 9)$ (see \cite[Web]{CP2026} for example). 
Firstly, since $\Ga$ is a connected $X$-arc-transitive cubic graph and $|V(\Sigma)|=30$, we conclude that $\Ga$ is a $X$-vertex-transitive locally-primitive graph and $K$ has 30 orbits on $V(\Ga)$ respectively. 
By Proposition \ref{NQ}(2), $\Ga$ is $(X, s)$-arc-transitive if and only if $\Sigma$ is $(X/K, s)$-arc-transitive, where positive integer $s\leqslant5$, $K\unlhd X$ and $K$ has at least three orbits on $V(\Ga)$. 
Let $Y^+$ be the subgroup of $Y$ that fixes each part of the bipartition of $V(\Sigma)$ and $X^+=K.Y^+$. 
Since $\Ga$ is $X$-arc-transitive, we get that $\Sigma$ is $Y$-arc-transitive, which means that $90\mid |Y|$ and $Y^+$ acts transitively on each partition. 
Thus $|X:X^+|=|Y:Y^+|=2$, $X_\a=X^+_\a$ for $\a\in V(\Ga)$, and $Y_v=Y^+_v$ for $v\in V(\Sigma)$. 
Noting that $Y\leqslant\Aut(\Sigma)\cong\PGammaL(2, 9)$, it follows that $Y\in\{\PSL(2, 9),\,\PGL(2, 9),\,\M_{10},\,\Sy_6,\,\PGammaL(2, 9)\}$. Obviously, $Y\not\cong\PSL(2,9)$ as $|Y:Y^+|=2$. 
By \cite[Page 4]{C1985}, $\Sy_6$ has two orbits of length $15$ on $V(\Sigma)$. Thus $\Sy_6$ is not vertex-transitive on $\Sigma$, and so $Y\not\cong\Sy_6$.
Thus, $Y\in\{\PGL(2, 9),\,\M_{10},\,\PGammaL(2, 9)\}$. Next, we will discuss the two cases separately.

{\it Case 1.} If $Y\cong\PGL(2, 9)$ or $\M_{10}$, then $\Ga\cong{\rm F90A}$.

In this case, $|Y|=720$. The transitivity of $Y$ on $V(\Sigma)$, together with $|Y|=720$ and $|V(\Sigma)|=30$, gives $|Y_v|=24$. The classification in \cite{DM1980} shows that $\Sigma$ is $(Y,4)$-arc-regular, with $Y_v\cong\Sy_4$. By \cite[Page 4]{C1985}, $Y^+\cong\PSL(2,9)\cong\A_6$. Hence $X^+=K.Y^+\cong K.\PSL(2,9)$ and $X_\a=X^+_\a\cong\Sy_4$.

Since $\Ga$ is $X$-vertex-transitive and locally primitive, Proposition \ref{NQ} gives that $\Ga$ is $(X,4)$-arc-transitive.
By Propositions \ref{simple} and \ref{schur}, we get $X^+=K.Y^+\cong\mz_n.\PSL(2, 9)$ is a central extension, $X^+=K(X^+)'$, and $(X^+)'=(K\cap (X^+)').\PSL(2,9)$ is perfect, where $K\cap (X^+)'\lesssim {\rm Mult}(\PSL(2, 9))\cong\mz_6$. 
It is easy for us to get that $(X^+)'\in\{\PSL(2, 9),\,\SL(2, 9),\,\mz_3.\PSL(2, 9),\,\mz_6.\PSL(2, 9)\}$.
Since $(X^+)'\unlhd X^+$, we have $(X^+)'_\a=(X^+)'\cap X^+_\a\unlhd X^+_\a$ and so $X^+_\a/(X^+)'_\a\cong X^+_\a(X^+)'/(X^+)'\leqslant X^+/(X^+)'=K(X^+)'/(X^+)'\cong K/(K \cap (X^+)')$. 
Noting that $K$ is cyclic and $X_\a^+\cong\S_4$, we conclude that $(X^+)'_\a\cong\A_4$ or $\Sy_4$. 
By \cite[Page~394, 6.3 (ii)]{S1982}, $\SL(2,9)\cong\mz_2.\PSL(2,9)$ has a unique involution, so it contains no subgroup isomorphic to $\A_4$. 

Assume that $\A_4\leqslant\mz_6.\PSL(2,9)$. Since $\mz_3\leqslant Z(\mz_6.\PSL(2,9))$, we have $\A_4\cap\mz_3\leqslant Z(\A_4)=1$. Hence $\A_4\cong\A_4\mz_3/\mz_3\leqslant(\mz_6.\PSL(2,9))/\mz_3\cong\mz_2.\PSL(2,9)\cong\SL(2,9)$, a contradiction.
Then $(X^+)'\in\{\PSL(2,9),\,\mz_3.\PSL(2,9)\}$.
Since $(X^+)'\char X^+\unlhd X$, we have $(X^+)'\unlhd X$ and so all orbits of $(X^+)'$ on $V(\Ga)$ have the same length. Noting that $(X^+)'$ is not transitive on $V(\Ga)$, we deduce from the Proposition \ref{NQ} that if $(X^+)'$ has at least three orbits on $V(\Ga)$, then $(X^+)'$ is semiregular on $V(\Ga)$. However, since $(X^+)'_\a\neq 1$, we conclude that $(X^+)'$ has exactly two orbits. 

Suppose that $(X^+)'\cong\PSL(2,9)$. 
Then $|\a^{(X^+)'}|=|\PSL(2,9):(X^+)'_\a|=30/m$ where $m=1$ or 2, depending on $(X^+)'_\a\cong\A_4$ or $\Sy_4$. Noting that $|V(\Ga)|=30|K|$ and $K\neq1$, we obtain that $(X^+)'$ has $(30m|K|)/30=m|K|$ orbits on $V(\Ga)$, which implies that $m|K|=2$. Hence, $|K|=2$ and so $|V(\Ga)|=60$. 
The Foster census \cite{CP2026} contains no $4$-arc-transitive cubic graph of order $60$, a contradiction.
Thus $(X^+)'\cong\mz_3.\PSL(2,9)$ and $|\mz_3.\PSL(2,9):(X^+)'_\a|=90/m$, where $m\in\{1,2\}$. Hence $(X^+)'$ has $m|K|/3$ orbits on $V(\Ga)$, and the equality $m|K|/3=2$ gives $|K|\in\{3,6\}$. The Foster census \cite{CP2026} contains no connected arc-transitive cubic graph of order $180$, so $|K|=3$ and $|V(\Ga)|=90$. The graph $\Ga$ is $(X,4)$-arc-transitive, and the Foster census identifies $\Ga$ with ${\rm F90A}$. This completes Case 1. 

{\it Case 2.} If $Y\cong\PGammaL(2, 9)$, then $\Ga\cong{\rm F90A}$.

Since $\Sigma$ is a $(Y, 5)$-transitive cubic graph, by \cite[Proposition 5]{DM1980}, we get that  $Y_v\cong\Sy_4\times\mz_2$, and so $Y^+_v=Y_v\cong X_\a=X^+_\a\cong\Sy_4\times\mz_2$, which implies that $\Ga$ is a connected 5-arc-transitive cubic graph. 
For $Y\cong\PGammaL(2, 9)$, we have $Y^+\cong\Sy_6,\,\PGL(2, 9)$ or $\M_{10}$. Since $|V(\Sigma)|=30$ and $|Y^+|=|v^{Y^+}||Y^+_v|=15|Y^+_v|$, $Y^+_v$ is index 15 subgroup of $Y^+$. However, there exists no subgroup of $\PGL(2, 9)$ or $\M_{10}$ with index 15 (see \cite[Page 4]{C1985} for example), which implies that $Y^+\cong\Sy_6$. Hence, $X^+=K.Y^+\cong K.\Sy_6$. 
Set $C=C_{X^+}(K)$. Then $K\leqslant C$ as $K$ is cycle.
By the $``N/C"$-theorem, $X^+/C\cong (X^+/K)/(C/K)\cong\Sy_6/(C/K)\lesssim\Aut(K)$. 
Since $\Aut(K)$ is abelian, we get that $C/K\cong\A_6$ or $\Sy_6$, which implies that  $|X^+:C|\leqslant2$.

Let $M\leqslant C$ such that $K\leqslant M$ and $M/K\cong \A_6$. We claim that $M'\in\{\A_6,\,\mz_3.\A_6\}$ and $M'$ has exactly two orbits on $V(\Ga)$. 
Since $|\a^M|=|M:M_\a|=|M|/|M_\a|$ and $|\a^{X^+}|=|X^+:X^+_\a|$, we deduce that $2=|X^+:M|=(|\a^{X^+}||X^+_\a|)/(|M_\a||\a^M|)$. 
Suppose that $|X^+_\a|=|M_\a|$. Then $2|\a^M|=|\a^{X^+}|=15|K|$. The group $M/K\cong\A_6$ is transitive on each part of $\Ga_K\cong\Sigma$, while $K$ acts regularly on each fibre of the cover. Thus $M$ is transitive on each part of $\Ga$, which gives $|\a^M|=15|K|$. This contradicts $2|\a^M|=15|K|$.
Thus, $|\a^M|=|\a^{X^+}|$ and $|X^+_\a:M_\a|=2$. 
On the other hand, since $X^+_\a\cong\Sy_4\times\mz_2$, we get that $M_\a$ is either $\Sy_4$ or $\A_4\times\mz_2$, and there exists a subgroup isomorphic to $\A_4$ in $M_\a$.
We deduce from Proposition \ref{simple} that $M=K.\A_6$ is a central extension, $M=KM'$ with $M'=(K\cap M').\A_6$ being a perfect group and $K\cap M'\lesssim {\rm Mult}(\PSL(2,9))\cong\mz_6$ (see the Proposition \ref{schur}) for example).
Hence, $M'\in\{\A_6,\,\mz_2.\A_6,\,\mz_3.\A_6,\,\mz_6.\A_6\}$.
Since $M/M'$ is abelian and $M'_\a\unlhd M_\a\cong\Sy_4$ or $\A_4\times\mz_2$, 
we get that $M'_\a\in\{\mz_2^2,\,\mz_2^3,\,\A_4,\,\A_4\times\mz_2,\,\Sy_4\}$, which implies that
there exists a subgroup isomorphic to $\mz_2\times\mz_2$ in $M'_\a$.
By \cite[Page 394, 6.3 (ii)]{S1982}, $2.\A_6\cong\SL(2,9)$ has a unique involution. Thus it contains no subgroup isomorphic to $\mz_2^2$. If $H\cong\mz_2^2$ is a subgroup of $6.\A_6$, then $H\cap\mz_3=1$. Thus $H\cong H\mz_3/\mz_3\leqslant(6.\A_6)/\mz_3\cong2.\A_6$, a contradiction. Then $M'\in\{\A_6,\,\mz_3.\A_6\}$.

The equality $|X^+:M|=2$ implies $(X^+)'\leqslant M$, so $(X^+)''\leqslant M'$. Also, $(X^+)''\char X^+\unlhd X$, which gives $(X^+)''\unlhd M'$. Suppose first that $M'\cong\A_6$. The subgroup $(X^+)''$ is unsolvable and normal in $M'$, so the simplicity of $\A_6$ forces $(X^+)''=M'$. Suppose next that $M'\cong\mz_3.\A_6$. In this case $Z(M')\cong\mz_3$ and $M'/Z(M')\cong\A_6$. 
The equality $|X^+:M|=2$ gives $(X^+)'\leqslant M$, so $(X^+)''\leqslant M'$. The relation $X^+/K\cong\Sy_6$ gives $(X^+)''K/K=(X^+/K)''\cong\A_6$. Thus $(X^+)''$ has a quotient isomorphic to $\A_6$ and is unsolvable. 
The unsolvability of $(X^+)''$ again forces $(X^+)''=M'$. Hence $M'=(X^+)''\unlhd X$, so all $M'$-orbits on $V(\Ga)$ have the same length.
Noting that $M'$ is not transitive on $V(\Ga)$ and $M'_\a\neq1$, we draw from Proposition \ref{NQ} that $M'$ has exactly two orbits. 

Suppose that $M'\cong\A_6$. The group $\A_6$ contains no subgroup isomorphic to $\mz_2^3$ or $\A_4\times\mz_2$, so $M'_\a\in\{\mz_2^2,\A_4,\Sy_4\}$. The equality $|\a^{M'}|=|M':M'_\a|$, together with the fact that $M'$ has exactly two orbits on $V(\Ga)$, gives $2=|V(\Ga)|/|\a^{M'}|=30|K||M'_\a|/360=|K||M'_\a|/12$.
Hence $|K|=6$, $2$ or $1$ according as $M'_\a\cong\mz_2^2$, $\A_4$ or $\Sy_4$. The case $|K|=1$ is impossible, leaving $|V(\Ga)|=180$ or $60$. The Foster census \cite{CP2026} contains no $5$-arc-transitive cubic graph of either order, a contradiction. Hence $M'\ncong\A_6$, and so $M'\cong\mz_3.\A_6$.

Let $H$ be a $2$-subgroup of $M'$. Since $\mz_3\leqslant Z(M')$, we have $H\cap\mz_3=1$, and hence $H\cong H\mz_3/\mz_3\leqslant M'/\mz_3\cong\A_6$.
A Sylow $2$-subgroup of $\A_6$ is dihedral of order $8$, so $M'$ contains no subgroup isomorphic to $\mz_2^3$. The group $\A_4\times\mz_2$ cannot occur either, since its Sylow $2$-subgroup is isomorphic to $\mz_2^3$. Thus $M'_\a\in\{\mz_2^2,\A_4,\Sy_4\}$. Using $|\a^{M'}|=|M':M'_\a|$ and the fact that $M'$ has exactly two orbits on $V(\Ga)$, we have $2=|V(\Ga)|/|\a^{M'}|=30|K|/|M':M'_\a|=|K||M'_\a|/36$.
Hence $|K|=18$, $6$ or $3$, and $|V(\Ga)|=540$, $180$ or $90$, according as $M'_\a\cong\mz_2^2$, $\A_4$ or $\Sy_4$.
The Foster census \cite{CP2026} contains no $5$-arc-transitive cubic graph of order $540$ or 180.
Hence $|K|=3$ and $|V(\Ga)|=90$, and the Foster census identifies $\Ga$ with ${\rm F90A}$. \qed

\begin{lemma}\label{F102A}
Suppose that $\Sigma\cong {\rm F102A}$. Then $K\cong\mz_2$, and $\Ga\cong {\rm F204A}$. 
Moreover, $X\cong\mz_2\times\PSL(2, 17)$ and $\Ga$ is a $(X,4)$-arc-transitive graph.
\end{lemma}
\pf\, Since $\Ga$ is a connected $X$-arc-transitive cubic graph, it is also $X$-vertex-transitive locally-primitive. By Proposition \ref{NQ} (2), for a positive integer $s\leqslant5$, $\Ga$ is $(X,s)$-arc-transitive if and only if $\Sigma$ is $(Y,s)$-arc-transitive. Recall that $\Sigma\cong{\rm F102A}$ is the Biggs--Smith graph, which is $4$-arc-regular and has full automorphism group $\Aut(\Sigma)\cong\PSL(2,17)$; see \cite[Web]{CP2026}. Thus $Y\leqslant\Aut(\Sigma)\cong\PSL(2,17)$. The $Y$-arc-transitivity of $\Sigma$ gives $306\mid|Y|$. The list of maximal subgroups of $\PSL(2,17)$ in \cite[Page 9]{C1985} then gives $Y=\Aut(\Sigma)\cong\PSL(2,17)$.
Since ${\rm F102A}$ is a $(Y,4)$-arc-regular cubic graph,
applying Proposition \ref{NQ} (2) again, we conclude that $\Ga$ is $(X,4)$-arc-regular, and for $v\in V(\Sigma)$ and $\a\in V(\Ga)$, we conclude that $Y_v\cong X_\a\cong\Sy_4$. 

Since $X=K.\PSL(2,17)$ where $K$ is cyclic, it follows from Propositions \ref{simple} and \ref{schur} that $X=KX'$ and $X'=(K\cap X').\PSL(2, 17)$ is perfect, with $K\cap X'\lesssim {\rm Mult}(\PSL(2, 17))\cong\mz_2$. 
This implies that $X'\in\{\PSL(2, 17),\,\SL(2, 17)\}$. 
As $X'\unlhd X$, we obtain that $X'_\a\unlhd X_\a$. 
Since $\Sy_4/X'_\a\cong X_\a/X'_\a= X_\a/(X'\cap X_\a)\cong X_\a X'/X'\leqslant X/X'=KX'/X'\cong K/(K\cap X')$ is cyclic, it follows that $X'_\a\cong \A_4$ or $\Sy_4$. 
Noting that $\SL(2, 17)$ contains only one involution (see \cite[Page 394 6.3 (ii)]{S1982} for example), one yields that it cannot contain any subgroup isomorphic to $\A_4$ or $\Sy_4$. Thus, $K\cap X'=1$ and $X'\cong\PSL(2, 17)$, and so $X=K\times X'\cong K\times\PSL(2, 17)$. Since all orbits of $X'$ on $V(\Ga)$ have the same length because $X'\unlhd X$, we deduce that $|\a^{X'}|=|X':X'_\a|=|\PSL(2, 17):12m|=204/m$ for $m=1$ or 2, and then $X'$ has $\frac{|V(\Ga)|}{|\a^{X'}|}=\frac{102|K|}{204/m}=m|K|/2$ orbits. 
Since $X'_\a\neq1$, by Proposition \ref{NQ}, we get that $X'$ has at most two orbits, that is, $m|K|/2\leqslant2$ where $m=1$ or 2. Hence, $|K|=2$ or 4, and so $|V(\Ga)|=204$ or 408. 
If $|V(\Ga)|=408$, then by \cite[Web]{CP2026}, there is no 4-arc-transitive cubic graph of order 408, a contradiction. Thus, $|V(\Ga)|=204$. 

The graph $\Ga$ is $(X,4)$-arc-regular and has order $204$, so the Foster census \cite{CP2026} gives $\Ga\cong{\rm F204A}$. Also, $\Ga$ is a $K$-cover of $\Sigma\cong{\rm F102A}$, and thus $\Ga_K\cong{\rm F102A}$. Example \ref{204V} identifies ${\rm F204A}$ as the standard double cover of ${\rm F102A}$. \qed

Based on the fundamental theories and methodological techniques currently at our disposal, we can now proceed to prove our Theorem  \ref{result 1}.

{\it \textbf{Proof of Theorem \ref{result 1}:}}
Suppose that $\Sigma\cong {\rm K_{3,3}},\,{\rm F30A}$ or ${\rm F102A}$. Then by \cite[Theorem 1.1]{FK20042}, Lemmas \ref{F30A} and \ref{F102A}, one can get that Theorem \ref{result 1} (1), (3) or (4) hold, respectively. 
Assume next that $\Sigma\cong{\rm DC_{14}}$. \cite[Theorem 1.1]{WH2010} classifies all connected arc-transitive regular cyclic covers of the Heawood graph. Rewriting the Cayley graphs in that classification in the notation of Example \ref{def:DCfamily}, the first family is $\mathcal D(1,7n,\lambda)$, where $\lambda^2+\lambda+1\equiv0\pmod{7n}$. The second family occurs when $7\mid n$ and is $\mathcal D(7,n/7,\mu)$, where $\mu=0$ for $n=7$ and $\mu^2+\mu+1\equiv0\pmod{n/7}$ for $n>7$. Conversely, the graphs in these two families are connected arc-transitive regular $\mz_n$-covers of ${\rm DC_{14}}$ by \cite[Theorem 1.1]{WH2010} and Example \ref{def:DCfamily}. This proves Theorem \ref{result 1} (2).  \qed

\section{\texorpdfstring{$\mz_p^2$-Covers}{Zp2-Covers}}\label{sec4}
In this section, let $\Ga$ be a connected $X$-arc-transitive regular cover of the edge-primitive cubic graph $\Sigma$, where $X$ is the fibre-preserving group and the covering transformation group is $K\cong\mz_p^2$ for some prime $p$.  Since $X$ is the fibre-preserving group, we have $K\unlhd X$. Let $Y=X/K$.  
Then $Y\leqslant\Aut(\Sigma)$, $X=K.Y$ and $X_\a\cong Y_v$ for any $\a\in V(\Ga)$ and $v\in V(\Sigma)$.  
With $\mz_p^2$-covers of ${\rm K_{3,3}}$ and ${\rm DC_{14}}$ already understood in \cite[Theorem 1.1]{DXY2018} and \cite[Theorem 7.1]{CM20132}, we now turn to a detailed study of the remaining two graphs. We treat these cases in the following two lemmas.

\begin{lemma}\label{F30A2}
The graph $\Sigma$ is not isomorphic to ${\rm F30A}$.
\end{lemma}
\pf\, Suppose that $\Sigma\cong {\rm F30A}$, which is bipartite.
Let $Y^+$ be the subgroup of $Y$ that fixes each part of the bipartition of $V(\Sigma)$, and $X^+=K.Y^+$. 
Since $\Ga$ is $X$-arc-transitive, we get that$\Sigma$ is $Y$-arc-transitive, which means that $Y^+$ acts transitively on each partition. 
It follows that $|X:X^+|=|Y:Y^+|=2$, $X_\a=X^+_\a$ for $\a\in V(\Ga)$, and $Y_v=Y^+_v$ for $v\in V(\Sigma)$. 
With similar arguments as Lemma \ref{F30A}, we see that $\Ga$ is an $X$-vertex-transitive locally-primitive graph and $Y\in\{\PGL(2, 9),\,\M_{10},\,\PGammaL(2, 9)\}$. Next, we deal it with the following two steps.

{\it Step 1: exclude $Y\cong\PGL(2, 9)$ or $\M_{10}$.}

Assume that $Y\cong\PGL(2, 9)$ or $\M_{10}$. Since $|Y|=720$ and $Y$ is transitive on $V(\Sigma)$, we have $|Y_v|=24$. By \cite{DM1980}, $\Sigma$ is $(Y,4)$-arc-regular and $Y_v\cong\Sy_4$. Moreover, \cite[Page 4]{C1985} gives $Y^+\cong\A_6$. Thus $Y^+_v=Y_v\cong X_\a\cong\Sy_4$ and $X^+=K.Y^+\cong K.\A_6$.

Let $C=C_{X^+}(K)$. By the $``N/C"$-theorem, $Y^+/(C/K)\cong (X^+/K)/(C/K)\cong X^+/C \lesssim\Aut(K)$. 
Since $K\cong\mz_p^2$, we obtain that $\Aut(K)\cong\GL(2, p)$.
Applying Proposition \ref{GL}, one yields that $C/K=Y^+\cong\A_6$.
Hence, $C=X^+\cong K.\A_6$ and so $K\leqslant Z(X^+)$. Since $Y^+\cong\A_6$ is a perfect group, we get that  $(X^+)'K/K\cong (X^+/K)'\cong \A_6\cong X^+/K$, and so $X^+=K(X^+)'$. Thus, it follows from $K\leqslant Z(X^+)$ that $(X^+)'=[X^+,X^+]=[K(X^+)',K(X^+)']=[(X^+)',(X^+)']=(X^+)''$ and $(X^+)'/(K\cap (X^+)')\cong (X^+)'K/K=X^+/K\cong\A_6$. Hence, $(X^+)'=(K\cap (X^+)').\A_6$ is a proper central extension. 
By Proposition \ref{schur}, we deduce that $K\cap (X^+)'\lesssim {\rm Mult}(\A_6)\cong\mz_6$. 
Noting that $K\cong\mz_p^2$, we get $K\cap (X^+)'\cong\mz_k$ for some $k\in\{1,\,2,\,3\}$. 
Hence, $(X^+)'\in\{\A_6,\,\SL(2, 9),\,\mz_3.\A_6\}$. 

Since $(X^+)'\unlhd X^+$, we have that $(X^+)'_\a\unlhd X^+_\a\cong\Sy_4$.
Noting that $X^+_\a/(X^+)'_\a=X^+_\a/(X^+_\a\cap (X^+)')\cong X^+_\a(X^+)'/(X^+)'\leqslant X^+/(X^+)'=K(X^+)'/(X^+)'\cong K/(K\cap (X^+)')$ is abelian as $K\cong\mz_p^2$, we deduce that $(X^+)'_\a\cong\A_4$ or $\Sy_4$. 
Noting that $\SL(2, 9)$ contains only one involution (see \cite[Page 394 6.3(ii)]{S1982} for example), one yields that it cannot contain any subgroup isomorphic to $\A_4$ or $\Sy_4$. Therefore, $(X^+)'\in\{\A_6,\,\mz_3.\A_6\}$. 
Noting that $(X^+)'\char X^+\unlhd X$, we get $(X^+)'\unlhd X$. 
Thus, all orbits of $(X^+)'$ on $V(\Ga)$ have the same length. 
Noting that $(X^+)'$ is not transitive on $V(\Ga)$ and $(X^+)'_\a\neq1$, by Proposition \ref{NQ}, we get that $(X^+)'$ has exactly two orbits. 
Since $K\cong\mz_p^2$, this yields $|K|=p^2\geqslant4$.
Suppose that $(X^+)'\cong\A_6$. 
Then $|\a^{(X^+)'}|=|\A_6:(X^+)'_\a|=30/m$ where $m=1$ or 2, depending on $(X^+)'_\a\cong\A_4$ or $\Sy_4$, and so the number of the orbits of $(X^+)'$ is $2=|V(\Ga)|/|\a^{(X^+)'}|=30|K|/(30/m)= m|K|$. Thus, $|K|=1$ or $2$, a contradiction. Suppose that $(X^+)'\cong \mz_3.\A_6$. 
Then $|\mz_3.\A_6:(X^+)'_\a|=90/m$, where $m\in\{1,2\}$. The subgroup $(X^+)'$ has exactly two orbits on $V(\Ga)$, so $2=30|K|/(90/m)=m|K|/3$. It follows that $|K|=6/m\in\{3,6\}$, whereas $K\cong\mz_p^2$ has order $p^2$. This is impossible.

{\it Step 2: exclude $Y\cong\PGammaL(2,9)$.}

Since ${\rm F30A}$ is $(\PGammaL(2,9),5)$-transitive, by \cite[Proposition 5]{DM1980}, we obtain that $Y_v\cong\Sy_4\times\mz_2$, and hence $Y^+_v=Y_v\cong X_\a=X^+_\a\cong\Sy_4\times\mz_2$, where $v\in V(\Sigma)$ and $\a\in V(\Ga)$. 
By Proposition \ref{NQ} (2), the graph $\Ga$ is $(X,5)$-arc-transitive. We deduce from $\Sigma$ is bipartite and $Y\cong\PGammaL(2,9)$ that $|Y:Y^+|=2$ and $Y^+\in\{\Sy_6,\,\PGL(2,9),\,\M_{10}\}$. 
Since $|V(\Sigma)|=30$, the equality $|Y^+|=|v^{Y^+}||Y^+_v|=15|Y^+_v|$ shows that $Y^+_v$ has index 15 in $Y^+$. Neither $\PGL(2,9)$ nor $\M_{10}$ contains a subgroup of index 15 (see \cite[Page 4]{C1985} for example); thus $Y^+\cong\Sy_6$ and so $X^+=K.Y^+\cong K.\Sy_6$. 
Let $C=C_{X^+}(K)$. The $``N/C"$-theorem gives $\Sy_6/(C/K)\cong (X^+/K)/(C/K)\cong X^+/C\lesssim\Aut(K)$. It follows from $K\cong\mz_p^2$ that $\Aut(K)\cong\GL(2,p)$. Quoting with the Proposition \ref{GL}, we yield that $C/K\cong\A_6$ or $\Sy_6$, and $|X^+:C|\leqslant 2$.

Let $M\leqslant C$ with $K\leqslant M$ and $M/K\cong\A_6$. The possibilities $C/K\cong\A_6$ or $\Sy_6$, together with $|X^+:C|\leqslant2$, give $|X^+:M|=2$. The group $Y^+\cong\Sy_6$ is transitive on each part of the bipartition of $\Sigma$, while $M/K\cong\A_6$ is normal of index $2$ in $Y^+$. Hence $M/K$ has at most two orbits on either part. These orbits have the same length, whereas each part has $15$ vertices, so $M/K$ is transitive on each part of $\Sigma$. The group $K$ acts regularly on every fibre, which gives $|\a^M|=|\a^{X^+}|=15|K|$. Together with $|\a^M|=|\a^{X^+}|$ and $M_\a=M\cap X^+_\a$, the equality $|X^+:M|=2$ gives $|X^+:M|=\frac{|X^+|}{|M|}=\frac{|\a^{X^+}||X^+_\a|}{|\a^M||M_\a|}=|X^+_\a:M_\a|=2$.
Since $X^+_\a\cong\Sy_4\times\mz_2$, the subgroup $M_\a$ is isomorphic to either $\Sy_4$ or $\A_4\times\mz_2$. In particular, $M_\a$ contains a subgroup isomorphic to $\A_4$.

The inclusion $K\leqslant Z(C)$ gives $K\leqslant Z(M)$. The quotient $M/K\cong\A_6$ is perfect, so $M'K/K=(M/K)'=M/K$ and $M=KM'$. Hence $M'=[KM',KM']=[M',M']=M''$, while $M'/(K\cap M')\cong\A_6$ and $K\cap M'\leqslant Z(M')$. The relations $M'=M''$, $K\cap M'\leqslant Z(M')$ and $M'/(K\cap M')\cong\A_6$ imply that $K\cap M'$ is isomorphic to a quotient of ${\rm Mult}(\A_6)$. By Proposition \ref{schur}, ${\rm Mult}(\A_6)\cong\mz_6$. As $K\cong\mz_p^2$, it follows that $K\cap M'\cong1$, $\mz_2$ or $\mz_3$, and hence $M'\in\{\A_6,\mz_2.\A_6,\mz_3.\A_6\}$.

Now $M'_\a=M'\cap M_\a\unlhd M_\a$, and $M_\a/M'_\a$ is abelian because it embeds in $M/M'$. Hence $M'_\a\in\{\mz_2^2,\mz_2^3,\A_4,\A_4\times\mz_2,\Sy_4\}$, so $M'_\a$ contains a subgroup isomorphic to $\mz_2^2$. By \cite[Page 394, 6.3 (ii)]{S1982}, $2.\A_6\cong\SL(2,9)$ has a unique involution and contains no subgroup isomorphic to $\mz_2^2$. Hence $M'\in\{\A_6,\mz_3.\A_6\}$.

The equality $|X^+:M|=2$ gives $(X^+)'\leqslant M$, so $(X^+)''\leqslant M'$. The relation $X^+/K\cong\Sy_6$ gives $(X^+)''K/K=(X^+/K)''\cong\A_6$. Thus $(X^+)''$ has a quotient isomorphic to $\A_6$ and is unsolvable. Also, $(X^+)''\char X^+\unlhd X$, which gives $(X^+)''\unlhd M'$. If $M'\cong\A_6$, the simplicity of $\A_6$ gives $(X^+)''=M'$. Suppose that $M'\cong\mz_3.\A_6$, and let $Z=Z(M')\cong\mz_3$. The subgroup $(X^+)''Z/Z$ is normal in $M'/Z\cong\A_6$. It is nontrivial because $(X^+)''$ is unsolvable, so $(X^+)''Z=M'$. Hence $M'/(X^+)''$ is a quotient of $Z$ and is abelian. The group $M'$ is perfect, so this quotient is also perfect and must be trivial. Thus $(X^+)''=M'$ in both cases. Hence $M'=(X^+)''\unlhd X$, and all $M'$-orbits on $V(\Ga)$ have the same length. The subgroup $M'$ is not transitive and $M'_\a\neq1$, so Proposition \ref{NQ} shows that $M'$ has exactly two orbits on $V(\Ga)$. 
Suppose that $M'\cong\mz_3.\A_6$. Then $M'\cap K\cong\mz_3$, so $|K|=9$ and $|V(\Ga)|=270$. The Foster census \cite{CP2026} contains no arc-transitive cubic graph of order $270$, a contradiction. Hence $M'\cong\A_6$. Since $\A_6$ contains no subgroup isomorphic to $\mz_2^3$, we have $M'_\a\in\{\mz_2^2,\A_4,\Sy_4\}$ and $|\a^{M'}|=360/|M'_\a|$. The two $M'$-orbits give $2=30|K|/(360/|M'_\a|)=|K||M'_\a|/12$. For $|M'_\a|\in\{4,12,24\}$ this yields $|K|\in\{6,2,1\}$, impossible because $|K|=p^2$. This proves Lemma \ref{F30A2}. \qed

\begin{lemma}\label{F102A2}
The graph $\Sigma$ is not isomorphic to ${\rm F102A}$.
\end{lemma}
\pf\, Suppose that $\Sigma\cong {\rm F102A}$. 
Arguing as in Lemma \ref{F102A}, we find that $\Ga$ is an $X$-vertex-transitive locally-primitive graph, $Y\cong\PSL(2,17)$ and $X=K.\PSL(2, 17)$ with $K\cong\mz_p^2$. 
Recall that $\Sigma\cong{\rm F102A}$ is the Biggs--Smith graph, which is $4$-arc-regular and has full automorphism group $\PSL(2,17)$; see \cite[Web]{CP2026}. Hence $Y_v\cong X_\a\cong\Sy_4$ for $v\in V(\Sigma)$ and $\a\in V(\Ga)$. Proposition \ref{NQ} (2) gives that $\Ga$ is $(X,4)$-arc-transitive.

Let $C=C_X(K)$. Then $K\leqslant Z(C)$. The $N/C$-theorem gives $X/C\lesssim\Aut(K)\cong\GL(2,p)$. Since $X/K\cong\PSL(2,17)$ is simple, the quotient $X/C$ is either trivial or isomorphic to $\PSL(2,17)$. Proposition \ref{GL} excludes the latter, so $C=X$. Since $X/K$ is perfect, $X'K/K=(X/K)'=X/K$, and hence $X=X'K$. The inclusion $K\leqslant Z(X)$ gives $K\cap X'\leqslant Z(X')$. Also, $X'=[X'K,X'K]=[X',X']=X''$, so $X'$ is perfect. Proposition \ref{schur} shows that $|K\cap X'|$ divides $|{\rm Mult}(\PSL(2,17))|=2$. Hence $X'\cong\PSL(2,17)$ or $\SL(2,17)$.

Since $X'\unlhd X$, we have $X'_\a\unlhd X_\a\cong\Sy_4$. Moreover, $X_\a/X'_\a\cong X_\a X'/X'\leqslant X/X'\cong K/(K\cap X')$ is abelian, so $X'_\a\cong\A_4$ or $\Sy_4$. Proposition \ref{NQ} shows that $X'$ has at most two orbits on $V(\Ga)$. The group $\SL(2,17)$ has a unique involution by \cite[Page~394, 6.3 (ii)]{S1982}, so it contains no subgroup isomorphic to $\A_4$ or $\Sy_4$. Thus $X'\cong\PSL(2,17)$ and $K\cap X'=1$, giving $X=K\times X'\cong K\times\PSL(2,17)$.

Now $|\a^{X'}|=|X':X'_\a|=204/m$, where $m=1$ or $2$, and hence $X'$ has $m|K|/2$ orbits on $V(\Ga)$. Since this number is at most two and $|K|=p^2$, we obtain $|K|=4$. Thus $|V(\Ga)|=408$, but the Foster census \cite{CP2026} contains no $4$-arc-transitive cubic graph of order $408$, a contradiction. This proves Lemma \ref{F102A2}.\qed

Based on the fundamental theories and methodological techniques currently at our disposal, we can now proceed to prove our Theorem \ref{result 2}.

{\it \textbf{Proof of Theorem \ref{result 2}:}}

Let $\Sigma$ be an edge-primitive cubic graph and let $K\cong\mz_p^2$. By \cite[Lemma~2.9]{W1973}, $\Sigma\cong{\rm K_{3,3}}$, ${\rm DC_{14}}$, ${\rm F30A}$ or ${\rm F102A}$. Lemmas \ref{F30A2} and \ref{F102A2} exclude the last two possibilities. It remains to consider ${\rm K_{3,3}}$ and ${\rm DC_{14}}$.

(I). First suppose that $\Sigma\cong{\rm K_{3,3}}$. If $p\in\{2,3\}$, then $|V(\Ga)|=6p^2$, and \cite[Web]{CP2026} gives $\Ga\cong{\rm F24A}$ for $p=2$ and $\Ga\cong{\rm F54A}$ for $p=3$. Now assume that $p\geqslant5$. The $Y$-arc-transitivity of $\Sigma$ gives $18\mid|Y|$. We can label the vertices in $\Sigma$ as $\{1,\,2,\,\dots,\,6\}$. Let us define the following permutations of this vertex set: $g_1=(1,3,5)$, $g_2=(2,4,6)$, $g_3=(1,6)(2,3)(4,5)$, $g_4=(1,6)(2,5,4,3)$ and $g_5=(2,4)(3,5)$. Set $Y_1=\l g_1,\,g_2,\,g_3\r$, $Y_2=\l g_1,\,g_2,\,g_4,\,g_5\r$, $Y_3=\l g_1,\,g_2,\,g_3,\,g_5\r$ and $Y_4=\Aut(\Sigma)\cong\Sy_3\wr\mz_2$. The generators show that $Y_1\cong\mz_3^2\rtimes\mz_2$ is $1$-arc-regular on $\Sigma$, while $Y_2\cong\mz_3^2\rtimes\mz_4$ and $Y_3\cong\Sy_3\times\Sy_3$ are $2$-arc-regular, and $Y_4$ is $2$-arc-transitive.

It remains to show that the above four groups exhaust all arc-transitive possibilities for $Y$. A Magma computation enumerating the conjugacy classes of subgroups of $\Aut(\Sigma)$ whose orders are divisible by $18$ shows that there are precisely four conjugacy classes of arc-transitive subgroups. Their representatives have orders $18,36,36$ and $72$, respectively, and are conjugate in $\Aut(\Sigma)$ to $Y_1,Y_2,Y_3$ and $Y_4$, respectively. Hence $Y_1,Y_2,Y_3$ and $Y_4$ exhaust all arc-transitive possibilities for $Y$ up to conjugacy in $\Aut(\Sigma)$; see Appendix \ref{app:K33}. After conjugating $Y$ in $\Aut(\Sigma)$ if necessary, we may assume that $Y=Y_i$ for some $i\in\{1,2,3,4\}$.

If $Y=Y_i$ for some $i\in\{2,3,4\}$, then $\Sigma$ is $(Y,2)$-arc-transitive. By Proposition \ref{NQ} (2), $\Ga$ is $(X,2)$-arc-transitive. Hence \cite[Theorem 1.1]{DXY2018} and Example \ref{cover1} give $\Ga\cong X(3)$.
It remains to consider $Y=Y_1$. In this case $\Sigma$ is $Y$-arc-regular. By Proposition \ref{NQ} (2), $\Ga$ is $X$-arc-transitive. Since $|X|=|K||Y|=18p^2$ equals the number of arcs of $\Ga$, the action of $X$ on the arcs is regular. Choose the spanning tree of $\Sigma$ shown in Figure \ref{K_{3,3}} and normalize the voltage assignment so that every tree arc has voltage $0$. Let $x_1,x_2,x_3,x_4$ be the voltages on the four cotree arcs indicated in Figure \ref{K_{3,3}}. Since $\Ga$ is connected, \cite[Page 46, Exercise 2.2.4]{KN2007} gives $K=\langle x_1,x_2,x_3,x_4\rangle$. For $p\geqslant5$, the discussion in the second paragraph of \cite[Page~78]{KN2007} shows that $\Ga$ is isomorphic to one of the graphs defined by the voltage assignments in rows 3 and 4 of \cite[Table 2.1]{KN2007}. Hence $\Ga\cong {\rm K_{3,3}}\times_{\ell^1}\mz_p^2$ or ${\rm K_{3,3}}\times_{\ell^2}\mz_p^2$, as defined in Example \ref{cover2}.

(II). Now suppose that $\Sigma\cong{\rm DC_{14}}$. By \cite[Theorem~7.1]{CM20132}, for each prime $p\neq7$, the graph $\mathcal{H}_p$ in Example \ref{def:Hp} is the unique arc-transitive regular $\mz_p^2$-cover of ${\rm DC_{14}}$, and it is $1$-arc-regular. Hence $\Aut(\mathcal{H}_p)\cong\mz_p^2.(\mz_7\rtimes\mz_6)$. In particular, $\mathcal{H}_2\cong{\rm F56A}$. It remains to consider $p=7$. By \cite[Theorem~7.1 (f)]{CM20132}, there are exactly two non-isomorphic arc-transitive regular $\mz_7^2$-covers of ${\rm DC_{14}}$: the $2$-arc-regular graph $\mathcal{H}_7\cong{\rm F686B}$ and the $1$-arc-regular graph ${\rm F686C}$. The same theorem also states that the latter admits a cyclic covering group of order $49$. Thus the second possibility in Theorem \ref{result 2} (2) is precisely ${\rm F686C}$. This proves Theorem \ref{result 2} (2).   \qed

\section{Further research}

Together with the known results for ${\rm K_{3,3}}$ and ${\rm DC_{14}}$, Theorems \ref{result 1} and \ref{result 2} give the classification, up to isomorphism of the covering graphs, of arc-transitive regular covers of edge-primitive cubic graphs with cyclic or $\mz_p^2$ covering transformation groups. This leads to the following problem.

\begin{problem}
Classify the arc-transitive regular covers of edge-primitive cubic graphs whose covering transformation groups are abelian or metacyclic.
\end{problem}

\appendix

\section{Arc-transitive subgroups of $\Aut(\K_{3,3})$}\label{app:K33}

The following Magma computation verifies that, up to conjugacy in $\Aut(\K_{3,3})$, there are precisely four arc-transitive subgroups, represented by $Y_1,Y_2,Y_3$ and $Y_4$.

\begin{lstlisting}[basicstyle=\ttfamily\small,breaklines=true,columns=fullflexible,caption={Arc-transitive subgroups of $\Aut(\K_{3,3})$.},label={code:K33}]
S6 := SymmetricGroup(6);

g1 := S6!(1,3,5);
g2 := S6!(2,4,6);
g3 := S6!(1,6)(2,3)(4,5);
g4 := S6!(1,6)(2,5,4,3);
g5 := S6!(2,4)(3,5);

Y1 := sub< S6 | g1, g2, g3 >;
Y2 := sub< S6 | g1, g2, g4, g5 >;
Y3 := sub< S6 | g1, g2, g3, g5 >;

a1  := S6!(1,3,5);
b1  := S6!(1,3);
a2  := S6!(2,4,6);
b2  := S6!(2,4);
tau := S6!(1,2)(3,4)(5,6);

A := sub< S6 | a1, b1, a2, b2, tau >;
Y4 := A;

Order(A);
[ Order(Y) : Y in [Y1,Y2,Y3,Y4] ];

IsArcTransitive := function(H)
    if #Orbit(H,1) ne 6 then
        return false;
    end if;
    H1 := Stabilizer(H,1);
    return #Orbit(H1,2) eq 3;
end function;

[ IsArcTransitive(Y) : Y in [Y1,Y2,Y3,Y4] ];

C := Subgroups(A : OrderMultipleOf := 18);

ArcSubs := [];
for R in C do
    H := R`subgroup;
    if IsArcTransitive(H) then
        Append(~ArcSubs,H);
    end if;
end for;

#ArcSubs;
[ Order(H) : H in ArcSubs ];

Ys := [Y1,Y2,Y3,Y4];

for H in ArcSubs do
    matches := [];
    for i in [1..4] do
        if IsConjugate(A,H,Ys[i]) then
            Append(~matches,i);
        end if;
    end for;
    <Order(H),matches>;
end for;

/* Output:
72
[ 18, 36, 36, 72 ]
[ true, true, true, true ]
4
[ 18, 36, 36, 72 ]
<18, [ 1 ]>
<36, [ 2 ]>
<36, [ 3 ]>
<72, [ 4 ]>
*/
\end{lstlisting}

\noindent{\bf Declaration of competing interest}

The authors declare that they have no known competing financial interests or personal relationships that could have appeared to influence the work reported in this paper.

\end{document}